\documentclass[preprint,11pt]{elsarticle}
\usepackage{amsthm}
\usepackage{amsmath}
\usepackage{amssymb}
\usepackage{enumerate}

\newcommand{\p}{\rho}

\newcommand{\abp}{\rho^{\mbox{ab}}}

\newcommand{\nats}{{\mathbb N}}
\newcommand{\ints}{{\mathbb Z}}

\newcommand{\No}{\mathbb N_0}
\newcommand{\set}[2]{ \left\{ #1 \mid #2 \right\} }

\newcommand{\kae}[1]{ \thicksim_{#1} }
\newcommand{\pref}[1]{\mathrm{pref}_{#1}}
\newcommand{\suff}[1]{\mathrm{suff}_{#1}}
\newcommand{\fc}[2]{ \mathcal{P}_{#2}^{(#1)}}
\newcommand{\swap}[1]{ g_{#1} }

\newtheorem{theorem}{Theorem}[section]
\newtheorem{lemma}[theorem]{Lemma}
\newtheorem{corollary}[theorem]{Corollary}
\newtheorem{proposition}[theorem]{Proposition}

\newtheorem{remark}[theorem]{Remark}
\newtheorem{example}[theorem]{Example}
\newtheorem{definition}[theorem]{Definition}

\begin{document}

\begin{frontmatter}

\title{
On a generalization of Abelian equivalence and complexity of infinite words}

\author[label4]{Juhani Karhumaki\corref{cor1}\fnref{label1}}
  \ead{karhumak@utu.fi}

 \author[label4]{Aleksi Saarela\fnref{label2}}
  \ead{amsaar@utu.fi}

   \author[label4,label5]{Luca Q. Zamboni\fnref{label3}}
  \ead{lupastis@gmail.com}

  \fntext[label1]{Partially supported by the Academy of Finland under grants
 251371 and 257857.}
  \fntext[label3]{Partially supported by a FiDiPro grant (137991) from the Academy of Finland and by
ANR grant {\sl SUBTILE}.}
\address[label4]{Department of Mathematics and Statistics \& FUNDIM,
University of Turku,
FI-20014 Turku, Finland}
\address[label5]{Universit\'e de Lyon,
Universit\'e Lyon 1, CNRS UMR 5208,
Institut Camille Jordan,
43 boulevard du 11 novembre 1918,
F69622 Villeurbanne Cedex, France}

\begin{abstract}
In this paper we introduce and study a family of complexity functions of infinite words  indexed by $k \in \ints ^+ \cup \{+\infty\}.$
Let $k \in \ints ^+ \cup \{+\infty\}$ and $A$ be a finite non-empty set. Two finite words $u$ and $v$ in $A^*$ are said to be $k$-Abelian equivalent if for all $x\in A^*$ of length less than or equal to $k,$ the number of occurrences of $x$ in $u$ is equal to the number of occurrences of $x$ in  $v.$ This defines a family of  equivalence relations $\thicksim_k$ on $A^*,$ bridging the gap between the usual notion of Abelian equivalence (when $k=1$) and equality (when $k=+\infty).$
We show that the number of $k$-Abelian equivalence classes of words of length $n$ grows polynomially, although the degree is exponential in $k.$ Given an infinite word $\omega \in A^\nats,$
we consider the associated complexity function $\mathcal {P}^{(k)}_\omega :\nats \rightarrow \nats$ which counts the number of  $k$-Abelian equivalence classes of factors of $\omega$ of length $n.$
We show that the complexity function $\mathcal {P}^{(k)}$
is intimately linked with periodicity. More precisely we define an auxiliary function $q^k: \nats \rightarrow \nats$ and show that  if $\mathcal {P}^{(k)}_{\omega}(n)<q^k(n)$ for some $k \in \ints ^+ \cup \{+\infty\}$ and $n\geq 0,$ the $\omega$ is ultimately periodic. Moreover  if $\omega$ is aperiodic, then $\mathcal {P}^{(k)}_{\omega}(n)=q^k(n)$ if and only if $\omega$ is Sturmian.
We also study $k$-Abelian complexity in connection with repetitions in words. Using Szemer\'edi's theorem, we show that if $\omega$ has bounded $k$-Abelian complexity, then for every $D\subset \nats$ with positive upper density and for every positive integer $N,$ there exists a $k$-Abelian $N$ power  occurring in $\omega$ at some position $j\in D.$ \end{abstract}


\end{frontmatter}
\section{Introduction}

Abelian equivalence of words has long been a subject of great interest (see for instance Erd\"os problem, \cite{CRSZ, CovHed, CurRam2009, Dekking1979, Keranen1992ICALP, PZ, RSZ1, RSZ2, aleksi}). Given a finite non-empty set $A,$  let $A^*$ denote the set of all finite words over $A.$  Two words $u$ and $v$ in $A^*$ are {\it Abelian equivalent,} denoted $u\thicksim_{\mbox{ab}} v,$ if and only if $|u|_a=|v|_a$ for all $a\in A,$ where $|u|_a$ and $|v|_a$ denote the number of occurrences of $a$ in $u$ and $v,$ respectively. It is readily verified that $\thicksim_{\mbox{ab}}$ defines an equivalence relation (in fact a congruence) on $A^*.$

We consider the following natural generalization: Fix  $k \in \ints ^+ \cup \{+\infty\}.$ Two  words $u$ and $v$ in $A^*$ are said to be $k$-{\it Abelian equivalent}, written $u\thicksim_k v,$ if $|u|_x=|v|_x$ for each non-empty word $x$ with $|x|\leq k$   (where $|x|$ denotes the length of $x,$ and $|u|_x$ and $|v|_x$ denote the number of occurrences of $x$ in $u$ and $v,$ respectively).
We note that  $u\thicksim _{+\infty} v$ if and only if $u=v,$ while  $\thicksim_1$ corresponds to the usual notion of Abelian equivalence $\thicksim_{\mbox{ab}}.$  Thus one may regard the notion of $k$-Abelian equivalence as gradually bridging the gap between Abelian equivalence ($k=1$) and equality ($k=+\infty).$  It is readily verified that $\thicksim_k$ defines an equivalence relation (in fact a congruence) on $A^*.$  Clearly,  if $u\thicksim_k v,$ then $|u|=|v|$ and $u\thicksim_{\ell} v$ for each positive integer $\ell \leq k.$

The notion of $k$-Abelian equivalence was first introduced by the first author in \cite{Ka80}
in connection with formal languages and  decidability questions of various fundamental problems.
It was shown that the well known Parikh Theorem on the equivalence of Parikh images of regular and context-free languages does not hold for $k$-abelian equivalence.
In contrast various highly nontrivial decidability questions including the D0L sequence equivalence problem \cite{ER} or the Post Correspondence Problem \cite{Post}, turned out to be easily decidable in the context of
$k$-Abelian equivalence. Recently $k$-Abelian equivalence has been studied in the context of avoidance of repetitions in words (see  the discussion at the beginning of \S\ref{last} on $k$-Abelian powers).
In this paper we undergo an investigation of the complexity of infinite words in the framework of $k$-Abelian equivalence. As is the case with various other notions of complexity of words, we will see that $k$-Abelian complexity is intimately linked with periodicity  and can be used to detect the presence of repetitions.

Let $A$ be a finite non-empty set. For each infinite word
$\omega= a _0a_1 a_2\ldots $ with $a_i\in A,$
we denote by ${\mathcal F}_{\omega}(n)$ the set of all {\it factors} of $\omega$ of length $n,$ that is, the set of all finite words of the form $a_{i}a_{i+1}\cdots a_{i+n-1}$ with $i\geq 0.$
We set \[\p_{\omega}(n)=\mbox{Card}({\mathcal F}_{\omega}(n)).\] The function $\p_{\omega}:\nats \rightarrow \nats$ is called the {\it factor complexity function} of $\omega.$
Analogously, for each $k \in \ints ^+ \cup \{+\infty\}$ we define  \[\mathcal {P}^{(k)}_\omega (n)=\mbox{Card}\left({\mathcal F}_{\omega}(n)/\thicksim_{k}\right).\]
The function
$\mathcal {P}^{(k)}_\omega :\nats \rightarrow \nats,$ which counts the number of $k$-Abelian equivalence classes of factors of $\omega$ of length $n,$ is called the $k$-{\it Abelian complexity} of $\omega.$ In case $k=+\infty$ we have that $\mathcal {P}^{(+\infty)}_\omega (n)=\p_{\omega}(n),$ while if $k=1,$
$\mathcal {P}^{(1)}_\omega (n),$ denoted $\abp_{\omega}(n),$ corresponds to the usual Abelian complexity of $\omega.$

Most word  complexity functions, including factor complexity \cite{MorHed1940}, maximal pattern complexity \cite{KZ}, permutation complexity \cite{AFKS, FDFF}, Abelian complexity \cite{CovHed}, and Abelian maximal pattern complexity \cite{KWZ}, may be used to detect
(and in some cases characterize) ultimately  periodic words. For instance,  a celebrated result  due to Morse and Hedlund \cite{MorHed1940} states that an infinite word $\omega \in A^\nats$ is ultimately periodic if and only if $\p_{\omega}(n)\leq n$ for some $n\in \ints^+.$ The third author together with T. Kamae proved a similar result in the context of maximal pattern complexity with $n$ replaced by $2n-1$ (see \cite{KZ}).
Furthermore,  amongst all aperiodic (meaning non-ultimately periodic) words, Sturmian words generally have the lowest possible complexity\footnote{ With respect to  maximal pattern complexity, and Abelian maximal pattern complexity, Sturmian words are not the only words of lowest  complexity.}. We  show that these same results hold in the framework of  $k$-Abelian complexity.
In order to formulate the precise link between aperiodicity and $k$-Abelian complexity, we define, for each $k \in \ints ^+ \cup \{+\infty\},$ an auxiliary function $q^{(k)}:\nats \rightarrow \nats$ by

\[ q^{(k)}(n) =
\left\{\begin{array}{ll} n+1 \,\,\,&\mbox{for}\,\, n\leq 2k-1\\
2k \,\,\,&\mbox{for}\,\, n\geq 2k
\end{array}
\right.
\]
We prove that  for $\omega \in A^\nats$,  if $\mathcal{P}^{(k)}_\omega (n_0)<q^{(k)}(n_0)$ for some $k \in \ints ^+ \cup \{+\infty\}$ and $n_0 \geq 1,$ then $\omega$ is ultimately periodic.

This result is already well  known in the special cases $k=+\infty$ and $k=1$ (see \cite{MorHed1940} and \cite{CovHed} respectively).  By the Morse-Hedlund result mentioned earlier,  this condition gives a characterization of ultimately periodic words in the special case $k=+\infty.$
In contrast, $k$-Abelian complexity does not yield such a characterization. Indeed, both Sturmian words and the ultimately periodic word  $01^\infty = 0111\cdots$ have the same constant $2$ Abelian complexity. More generally, we shall see that the ultimately periodic word
$0^{2k-1}1^\infty$  has the same $k$-Abelian complexity as a Sturmian word.
Nevertheless $k$-Abelian complexity gives a complete characterization of Sturmian words amongst all aperiodic words. More precisely, we prove that for an aperiodic word $\omega \in A^\nats,$  the following conditions are equivalent:
\begin{itemize}
\item  $\omega$ is a balanced binary word, that is, {\it Sturmian}.
\item $\mathcal {P}^{(k)}_\omega (n)=q^{(k)}(n)$ for each $k \in \ints ^+ \cup \{+\infty\}$ and $n\geq 1.$
\end{itemize}
Again, the special cases of $k=+\infty$ and $k=1$ were already known (see
\cite{MorHed1940} and \cite{CovHed} respectively).

Finally we investigate the question of avoidance of $k$-Abelian $N$ powers: By a $k$-Abelian $N$ power we mean a word $U$ of the form $U=U_1U_2\ldots U_N$ such that $U_i\thicksim_k U_j$ for all $1\leq i,j\leq N.$ Using Szemer\'edi's theorem \cite{Sz}, we show that if $\omega$ has bounded $k$-Abelian complexity, then for every $D\subset \nats$ with positive upper density and for every positive integer $N,$ there exists a $k$-Abelian $N$ power  occurring in $\omega$ at some position $j\in D.$

The paper is organized as follows: In \S\ref{back} we recall some basic definitions and notation and establish various basic properties of $k$-Abelian equivalence of words. Also in \S\ref{back} we compute the rate of growth of the number of $k$-Abelian equivalence classes of words in $A^n.$
In \S\ref{periodicity} we develop the link between $k$-Abelian complexity and periodicity of words. In \S\ref{Sturmwords} we compute the $k$-Abelian complexity of Sturmian words and show that it completely characterizes Sturmian words amongst all aperiodic words. Finally in \S\ref{last} we study $k$-Abelian complexity in the context of repetitions in words.

\section{$k$-Abelian equivalence}\label{back}

\subsection{Definitions and first properties}

Given a finite non-empty set $A,$ we denote by $A^*$ the set of all finite words over $A$ including the empty word, denoted by $\varepsilon,$ by $A^+$ the set of all finite non-empty words over $A,$ by $A^\nats$ the set of (right) infinite words over $A,$ and by
$A^\ints$ the set of  bi-infinite words over  $A.$
Given a finite word $u =a_1a_2\ldots a_n$ with $n \geq 1$ and $a_i \in A,$ we denote the length $n$ of $u$ by $|u|$ (by convention we set  $|\varepsilon|=0.)$   For each $x\in A^+,$ we let $|u|_x$  denote the number of occurrences of  $x$ in $u.$ For $u\in A^*,$ we denote by $\bar u$ the reverse of $u.$

A factor $u$ of $\omega=a_0a_1a_2\ldots \in A^\nats$ is called {\it right special} (respectively {\it left special}) if there exists distinct symbols $a,b\in A$ such that
both  $ua$ and $ub$ (respectively $au$ and $bu$) are factors of $\omega.$ We say $u$ is {\it bispecial} if $u$ is both left and right special.
An infinite word $\omega\in A^\nats$  is said to be \emph{periodic}  if there exists a positive integer $p$ such that
$a_{i+p} = a_i$ for all indices $i.$ It is said to be \emph{ultimately periodic} if $a_{i+p} = a_i$ for all sufficiently large $i$.
It is said to be \emph{aperiodic} if it is not ultimately periodic.
Sturmian words are the {\it simplest} aperiodic infinite words;
Sturmian words are infinite words over a binary alphabet having exactly $n+1$ factors of length
$n$ for each $n \geq 0.$ Their origin can be traced back to the astronomer J. Bernoulli  III in 1772.  A fundamental result due to Morse and Hedlund \cite{MorHed1940} states that each aperiodic (meaning non-ultimately periodic) infinite word must contain at least $n+1$ factors of each length $n\geq 0.$  Thus Sturmian words are those aperiodic words of lowest factor complexity.  They arise naturally in many different areas of mathematics including combinatorics, algebra, number theory, ergodic theory, dynamical systems and differential equations. Sturmian words are also of great importance in theoretical physics and in theoretical computer science and are used in
computer graphics as digital approximation of straight lines.
If $\omega \in \{a,b\}^\nats$ is Sturmian, then for each positive integer $n$ there exists a unique right special (respectively left special) factor of length $n,$ and one is the reversal of the other. In particular, if $x$ is a bispecial factor, the $x$ is a {\it palindrome}, i.e., $x =\bar x.$ For more on Sturmian words, we refer the reader to \cite{Lothaire1983book}.

%

\begin{definition}\label{df}\rm {Let $k \in \ints ^+ \cup \{+\infty\}.$ We say two words $u,v\in A^+$ are $k$-{\it Abelian equivalent} and write $u\thicksim_k v,$ if $|u|_x=|v|_x$ for all words $x$ of length $|x|\leq k.$ }
\end{definition}

\noindent We note that if $u,v\in A^+$ and $|u|=|v|\leq k,$ then $u\thicksim_k v$ if and only if $u=v.$

\begin{example}The words $u=010110$ and $v=011010$ are $3$-Abelian equivalent but not $4$-Abelian equivalent since the prefix $0101$ of $u$ does not occur in $v.$ The words $u=0110$ and $v=1101$ are not $2$-Abelian equivalent (since they are not Abelian equivalent) yet for every word $x$ of length $2$ we have $|u|_x=|v|_x.$
\end{example}

The next lemma
gives different equivalent ways of defining $k$-Abelian equivalence.
For example,
item \eqref{item:lessk} corresponds to the Definition~\ref{df}
and item \eqref{item:prefsuff} corresponds to another common definition:
Words $u$ and $v$ of length at least $k - 1$ are $k$-Abelian equivalent
if they share the same prefixes and suffixes of length $k - 1$ and
if $|u|_x = |v|_x$ for every word $t$ of length $k$.

\begin{lemma} \label{prefixsuffix}
Let $u$ and $v$ be words of length at least $k - 1$ and
let $|u|_t = |v|_t$ for every word $t$ of length $k$.
The following are equivalent:
\begin{enumerate}
\item \label{item:lessk}
  $|u|_s = |v|_s$ for all $s \in A^{\leq k - 1}$,
\item \label{item:k-1}
  $|u|_s = |v|_s$ for all $s \in A^{k - 1}$,
\item \label{item:prefsuff}
  $\pref{k - 1}(u) = \pref{k - 1}(v)$ and
  $\suff{k - 1}(u) = \suff{k - 1}(v)$,
\item \label{item:pref}
  $\pref{k - 1}(u) = \pref{k - 1}(v)$,
\item \label{item:suff}
  $\suff{k - 1}(u) = \suff{k - 1}(v)$,
\item \label{item:prefsuffi}
  $\pref{i}(u) = \pref{i}(v)$ and
  $\suff{k - 1 - i}(u) = \suff{k - 1 - i}(v)$
  for some $i \in \{0, \dots, k - 1\}$.
\end{enumerate}
\end{lemma}

\begin{proof}
\eqref{item:lessk} $\Rightarrow$ \eqref{item:k-1}:
Clear.

\eqref{item:k-1} $\Rightarrow$ \eqref{item:prefsuff}:
Let $\{t_1, \dots, t_n\}$ be the multiset
of factors of $u$ (and of $v$) of length $k$.
The multiset of factors of $u$ of length $k - 1$ is
\begin{equation*}
  \{\pref{k - 1}(u)\} \cup \{\suff{k - 1}(t_1), \dots, \suff{k - 1}(t_n)\},
\end{equation*}
and the multiset of factors of $v$ of length $k - 1$ is
\begin{equation*}
  \{\pref{k - 1}(v)\} \cup \{\suff{k - 1}(t_1), \dots, \suff{k - 1}(t_n)\}.
\end{equation*}
These multisets must be the same, so $\pref{k - 1}(u) = \pref{k - 1}(v)$.
Similarly, $\suff{k - 1}(u) = \suff{k - 1}(v)$.

\eqref{item:prefsuff} $\Rightarrow$ \eqref{item:pref}, \eqref{item:suff}:
Clear.

\eqref{item:pref} or \eqref{item:suff} $\Rightarrow$ \eqref{item:prefsuffi}:
Clear.

\eqref{item:prefsuffi} $\Rightarrow$ \eqref{item:lessk}:
Let $\{t_1, \dots, t_n\}$ be the multiset
of factors of $u$ (and of $v$) of length $k$.
Every
\begin{equation*}
  s \in A^{k - 1} \smallsetminus \{\pref{k - 1}(u), \suff{k - 1}(u)\}
\end{equation*}
appears in the multiset
\begin{equation} \label{eq:prefsuffmultiset}
  \{\pref{k - 1}(t_1), \dots, \pref{k - 1}(t_n)\} \cup
  \{\suff{k - 1}(t_1), \dots, \suff{k - 1}(t_n)\}
\end{equation}
$2 |u|_s$ times.
A word $s \in \{\pref{k - 1}(u), \suff{k - 1}(u)\}$
appears $2 |u|_s - 1$ times if $\pref{k - 1}(u) \ne \suff{k - 1}(u)$,
and $2 |u|_s - 2$ times if $\pref{k - 1}(u) = \suff{k - 1}(u)$.
Similarly, every
\begin{equation*}
  s \in A^{k - 1} \smallsetminus \{\pref{k - 1}(v), \suff{k - 1}(v)\}
\end{equation*}
appears $2 |v|_s$ times, and
a word $s \in \{\pref{k - 1}(v), \suff{k - 1}(v)\}$
appears $2 |v|_s - 1$ times if $\pref{k - 1}(v) \ne \suff{k - 1}(v)$,
and $2 |v|_s - 2$ times if $\pref{k - 1}(v) = \suff{k - 1}(v)$.

If some words appear an odd number of times in \eqref{eq:prefsuffmultiset},
then these must be $\pref{k - 1}(u)$ and $\suff{k - 1}(u)$,
and they must also be $\pref{k - 1}(v)$ and $\suff{k - 1}(v)$.
If follows that $|u|_s = |v|_s$ for every $s \in A^{k - 1}$.
(In this case the assumption \eqref{item:prefsuffi} was not needed.)

If all words appear an even number of times in \eqref{eq:prefsuffmultiset},
then necessarily
$\pref{k - 1}(u) = \suff{k - 1}(u)$ and $\pref{k - 1}(v) = \suff{k - 1}(v)$.
From \eqref{item:prefsuffi} it follows that
$\pref{k - 1}(u) = \pref{k - 1}(v)$ and $\suff{k - 1}(u) = \suff{k - 1}(v)$,
and thus $|u|_s = |v|_s$ for every $s \in A^{k - 1}$.

The fact that $|u|_s = |v|_s$ also for every $s$ of length less than $k - 1$
can be proved in a similar way.
\end{proof}

\noindent The next lemma lists some basic facts on $k$-Abelian equivalence:

\begin{lemma} \label{lem:basic}
Let $u, v \in A^*$ and $k \geq 1$.
\begin{itemize}
\item If $|u| = |v| \leq 2 k - 1$ and $u \kae{k} v$, then $u = v$.
\item If $u \kae{k} v$, then $u \kae{k'} v$ for all $k' \leq k$.
\item If $u_1 \kae{k} v_1$ and $u_2 \kae{k} v_2$,
  then $u_1 u_2 \kae{k} v_1 v_2$.
\end{itemize}
\end{lemma}

The bound $2k-1$ in Lemma~\ref{lem:basic} is optimal as for each positive integer $k$ there exist words $u\neq v$ of length $2k$ such that
$u\thicksim_{k} v.$ For example, the words $u=0^{k-1}010^{k-1}$ and $v=0^{k-1}100^{k-1}$ of length $2k$ are readily verified to be $k$-Abelian equivalent (see Proposition~\ref{simplificationprop}).

\begin{lemma}\label{centralword} Fix $2\leq k<+\infty.$ Suppose $aub\thicksim_{k} cvd$ with $a,b,c,d\in A$ and $u,v\in A^*.$ Then $u\thicksim_{k-1}v.$
\end{lemma}

\begin{proof} Let $x\in A^*$ with $|x|\leq k-1.$ We can assume that $|x|<|aub|$ for otherwise $0=|u|_x=|v|_x.$ If $x$ is neither a prefix nor a suffix of $aub,$ then by Lemma~\ref{prefixsuffix} $x$ is neither a prefix nor suffix of $cvd$ and hence $|u|_x=|aub|_x=|cvd|_x=|v|_x.$   If $x$ is either a prefix of $aub$ or a suffix of $aub$ but not both, the $|u|_x=|aub|_x-1=|cvd|_x-1=|v|_x.$ Finally if $x$ is both a prefix and a suffix of $aub$ then
$|u|_x=|aub|_x-2=|cvd|_x-2=|v|_x.$
\end{proof}

%

\subsection{A first connection to Sturmian words}

\noindent The next theorem gives a complete classification of pairs of $k$-Abelian equivalent words of length $2k$ and establishes a first link to Sturmian words:

\begin{theorem}\label{2k} Fix a positive integer $k,$ and let $u,v\in A^*$ be distinct words of length $2k.$ Then $u\thicksim_{k}v$ if and only if there exist distinct letters $a,b\in A,$ a Sturmian word $\omega \in \{a,b\}^\nats$ and a right special factor $x$ of $\omega$ of length $k-1$ (or empty in case $k=1)$ such that
\[u=xab\bar x\,\,\,\,\,\,\mbox{and} \,\,\,\,\,\, v=xba\bar x.\]
In particular $u$ and $v$ are both factors of the same Sturmian word $\omega.$\end{theorem}

\begin{remark}\label{bispecial}\rm{It follows that if $u$ and $v$ are distinct $k$-Abelian equivalent words of length $2k,$ then both $u$ and $v$ are on a binary alphabet and in fact factors of the same Sturmian word $\omega.$  In fact, if  $B$ is a bispecial factor of $\omega$ then both $BabB$ and $BbaB$ are factors of $\omega.$
Also, if $x$ is a right special factor of $\omega,$ then there exists a bispecial factor $B$ of $\omega$ with $x$ a suffix of $B$ and $\bar x$ a prefix of $B.$
Thus both $xab\bar x$ and $xba \bar x$ are factors of $\omega.$}
\end{remark}

\noindent We will need the next result applied to Sturmian words, but we prove it more generally for episturmian words. We refer the reader to \cite{DJP} for the definition and basic properties of episturmian words.

\begin{proposition}\label{simplificationprop} Fix a positive integer $k\geq 2.$ Let $u$ and $v$ be factors of the same episturmian word $\omega$. Then $u$ and $v$ are $k$-Abelian equivalent if and only if $u$ and $v$ are $(k-1)$-Abelian equivalent and share a common prefix and a common suffix of length $\mbox{min}\{|u|,k-1\}.$  Thus, $u$ and $v$ are $k$-Abelian equivalent if and only if $u$ and $v$ are Abelian equivalent and share a common prefix and a common suffix of length   $\mbox{min}\{|u|,k-1\}.$

\end{proposition}

\begin{proof} One direction follows immediately from Lemma~\ref{prefixsuffix}. Next suppose that $u$ and $v$ are $(k-1)$-Abelian equivalent factors of the same episturmian word $\omega,$  and that $u$ and $v$ share a common prefix and a common suffix of length $\mbox{min}\{|u|, k-1\}.$ To prove that $u\thicksim _k v$ it suffices to show that whenever $axb\in {\mathcal F}_{\omega}(k)$ (with $a,b \in A$ and $x\in A^*),$ we have $|u|_{axb}=|v|_{axb}.$
First let us suppose that $ax$ is not a right special factor of $\omega$ so that every occurrence in $\omega$ of $ax$ is a occurrence of $axb.$ Then, if $ax$ is not a suffix of $u$ (and hence not a suffix of $v)$ we obtain
\[|u|_{axb} =|u|_{ax}=|v|_{ax}=|v|_{axb}.\]
On the other hand if $ax$ is a suffix of $u$ (and hence also a suffix of $v)$ we have
\[|u|_{axb}=|u|_{ax}-1=|v|_{ax}-1=|v|_{axb}.\]
Similarly, in case $xb$ is not a left special factor of $\omega$ we obtain $|u|_{axb}=|v|_{axb}.$
Thus it remains to consider the case when $ax$ is right special in $\omega$ and $xb$ is left special in $\omega.$
In this case $x$ is bispecial and $a=b.$ For each $c\in A,$ let $n_c=|u|_{axc}$ and $n'_c=|v|_{axc}.$ We must show that $n_a=n'_a.$ However we know that $n_c=n'_c$ for all $c\neq a$ since $xc$ is not left special in $\omega.$
Now, if $ax$ is not a suffix of $u$ (and hence not a suffix of $v)$ we have
\[ \sum _{c\in A} n_c = |u|_{ax} =|v|_{ax} = \sum _{c\in A} n'_c\]
whence $n_a=n'_a.$ On the other hand if $ax$ is a suffix of $u$ (and hence a suffix of $v)$ then
\[ \sum _{c\in A} n_c = |u|_{ax}-1 =|v|_{ax}-1 = \sum _{c\in A} n'_c\]
whence $n_a=n'_a$ as required.
\end{proof}

\begin{remark} \rm{The following example illustrates that the assumption in Proposition~\ref{simplificationprop} that $u$ and $v$ are factors of the same Sturmian word is necessary: Let $u=aabb$ and $v=abab.$ The $u$ and $v$ are Abelian equivalent and share a common prefix and suffix of length $1,$ yet they are not $2$-Abelian equivalent.}
\end{remark}

\begin{proof}[Proof of Theorem~\ref{2k}]
We start by showing that if $\omega \in \{a,b\}^\nats$ is a Sturmian word, and $x$ a right special factor of $\omega$ of length $k-1,$ then  $u=xab\bar x $ and $v= xba \bar x$ are $k$-Abelian equivalent. This follows from Proposition~\ref{simplificationprop} since $u$ and $v$ share a common prefix and a common suffix of lengths $k-1$ and are Abelian equivalent.

Next we suppose that $u$ and $v$ are distinct $k$-Abelian equivalent words of length $2k$ and show that both $u$ and $v$ have the required form.
We  proceed  by induction on $k.$ In case $k=1,$ we have that $u$ and $v$ are distinct Abelian equivalent words of length $2$ whence $u$ and $v$ may be written in the form $u=ab$ and $v=ba$ for some $a\neq b$ in $A.$

Next suppose the result of Theorem~\ref{2k} is true for $k-1$ and we shall prove it for $k.$ So let $u$ and $v$ be distinct $k$-Abelian equivalent words of length $2k$ with $k>1.$ Then by Lemma~\ref{prefixsuffix} we can write $u=a'u'b'$ and $v=a'v'b'$ for some
$a',b'\in A$ and $u',v'\in A^*$ where $|u'|=|v'|= 2(k-1)\geq 2.$
Since $u$ and $v$ are distinct,  it follows that $u'\neq v'.$
Also, by Lemma~\ref{centralword} it follows that $u'\thicksim _{k-1} v'.$ Thus by induction hypothesis, there exist distinct letters $a,b\in A$ and a Sturmian word $\omega \in \{a,b\}^\nats$ such that $u'$ and $v'$ are both factors of $\omega$
of the form $u'=xab\bar x$ and $v'=xba \bar x$ for some right special factor $x$ of $\omega$ of length $k-2.$

Thus we can  write $u=a'xab\bar x b'$ and $v=a'xba\bar x b'.$ Since $u\thicksim _{k} v,$  $|a'xa|=k,$ and $a\neq b$ it follows that  $a'x$ must occur in $v'$ and hence $a' \in \{a,b\}.$ Similarly we deduce that $b'\in \{a,b\}.$

Let us first suppose that $x\neq \bar x.$ Then $a'xa$ must occur in $v'$ and $a\bar xb'$ must occur in $u'.$ Hence both $a'xa$ and $a\bar x b'$ are factors of $\omega.$ Moreover, since $x\neq \bar x$ it follows that $x$ is not left special in $\omega$ and $\bar x$ is not right special in $\omega.$ Hence every occurrence of $x$ in $\omega$ is preceded by $a'$ and every occurrence of $\bar x$ is $\omega$ is followed by $b'.$
Since the factors of $\omega$ are closed under reversal, we deduce that $a'=b'$ and $a'x$ is a right special factor of $\omega.$ Moreover, since
$u'$ and $v'$ are both factors of $\omega$ beginning in $x$ and ending in $\bar x,$ it follows that $u=a'xab\bar xa'$ and $v=a'xba\bar x a'$ are both factors of $\omega.$

Finally suppose $x=\bar x$ so that $x$ is a bispecial factor of $\omega.$ We may write the increasing sequence of bispecial factors $\varepsilon =B_0,B_1,\ldots ,x= B_n,B_{n+1},\ldots $ so that $x$ is the $n$th bispecial factor of $\omega.$  We recall that associated to $\omega$ is a sequence $(a_i)_{i\geq 0} \in A^\nats$ (called the {\it directive word }of $\omega)$ defined by $a_iB_i$ is right special in $\omega.$ (See for instance \cite{RiZa}).

Without loss of generality we can suppose that $a'=a.$ We claim $b'=a.$ Suppose to the contrary that $b'=b.$ Then both $axa$ and $b\bar x b=bxb$ are factors of $v'$ contradicting that $\omega$ is balanced.
Hence we must have $a'=b'=a$ and so
$u=axab\bar x a$ and $v=axba \bar x a.$ Now $x$ is a bispecial factor of the Sturmian word $\omega.$ If $ax$ is a right special factor of $\omega$ then we are done by Remark~\ref{bispecial}. Otherwise, if $bx$ is a right special factor of $\omega,$  then this means that $a_n=b$ where $a_n$ is the $n$th entry of the directive word of $\omega.$ Let $\omega'$ be a Sturmian word whose directive word $(b_i)_{i\geq 0} $ is defined by $b_i=a_i$ for $i\neq n,$ and $b_n=a.$ Then $x$ is a bispecial factor of $\omega'$ and $ax$ is a right special factor of $\omega'.$ It follows from Remark~\ref{bispecial} that both  $u$ and $v$ are factors of $\omega'.$
\end{proof}

\noindent As an immediate consequence of Theorem~\ref{2k} we have:

\begin{corollary} Let $u\in A^*$ be of the form $u=vxab\bar x w$ where $x$ is a right special factor of length $k-1$ of a Sturmian word. Set $u'= v x ba \bar x w.$
Then $u\thicksim_{k} u'.$
\end{corollary}

\subsection{The number of $k$-Abelian classes in $A^n$}

Here we shall estimate the number of $k$-Abelian equivalence classes of
words in $A^n.$ Fix
$k \geq 1$ and let $m \geq 2$ be the cardinality of the set $A.$

\begin{lemma} \label{lem:increasing}
The number of $k$-Abelian equivalence classes of $A^{n + 1}$
is at least as large as
the number of $k$-Abelian equivalence classes of $A^{n}$.
\end{lemma}

\begin{proof}
If $k = 1$ or $n < k - 1$, then the claim is clear.
Otherwise, let $B$ be a set of representatives
of the $k$-Abelian equivalence classes of $A^n$.
The set $A B$ has $m$ times as many words as $B$.
To prove the theorem, we will show that
there can be at most $m$ words in $A B$ that are $k$-Abelian equivalent.

Let $a \in A$ and
let $a u_0, \dots a u_m \in A B$ be $k$-Abelian equivalent.
It needs to be shown that some of these words are equal.
Two of these words must have the same $k$th letter,
let these be $au$ and $av$.
Because also $\pref{k - 1}(au) = \pref{k - 1}(av)$,
it follows that $\pref{k}(au) = \pref{k}(av)$.
If $t \in A^k$,
then either $|u|_t = |au|_t = |av|_t = |v|_t$
(if $t \ne \pref{k}(au)$),
or $|u|_t = |au|_t - 1 = |av|_t - 1 = |v|_t$
(if $t = \pref{k}(au)$).
Thus $u$ and $v$ are $k$-Abelian equivalent and,
by the definition of $B$, $u = v$.
This proves the claim.
\end{proof}

Let $s_1, s_2 \in A^{k - 1}$ and let
\begin{equation*}
  S(s_1, s_2, n) = A^n \cap s_1 A^* \cap A^* s_2
\end{equation*}
be the set of words of length $n$ that start with $s_1$ and end with $s_2$.
For every word $w \in S(s_1, s_2, n)$ we can define a function
\begin{equation*}
  f_w: A^k \to \{0, \dots, n - k + 1\}, \ f_w(t) = |w|_t.
\end{equation*}
If $u, v \in S(s_1, s_2, n)$, then $u \kae{k} v$ if and only if $f_u = f_v$.
To count the number of $k$-Abelian equivalence classes,
we need to count the number of the functions $f_w$.
Not every function
\begin{math}
  f: A^k \to \{0, \dots, n - k + 1\}
\end{math}
is possible.
It must be
\begin{equation} \label{eq:sumf}
  \sum_{t \in A^k} f(t) = n - k + 1,
\end{equation}
and there are also other restrictions,
which are determined in Lemma \ref{lem:euler}.

If a function
\begin{math}
  f: A^k \to \No
\end{math}
is given, then a directed multigraph $G_f$ can be defined as follows:
the set of vertices is $A^{k-1}$,
and if $t = s_1 a = b s_2$, where $a, b \in A$,
then there are $f(t)$ edges from $s_1$ to $s_2$.
If $f = f_w$, then this multigraph is related to the Rauzy graph of $w$.
In the next lemma,
$\deg^-$ denotes the indegree and $\deg^+$ the outdegree of a vertex in $G_f$.

\begin{lemma} \label{lem:euler}
For a function
\begin{math}
  f: A^k \to \No
\end{math}
and words $s_1, s_2 \in A^{k - 1}$, the following are equivalent:
\begin{enumerate}[(i)]
\item \label{euler1}
  there is a number $n$ and a word $w \in S(s_1, s_2, n)$
  such that $f = f_w$,
\item \label{euler2}
  there is an Eulerian path from $s_1$ to $s_2$ in $G_f$,
\item \label{euler3}
  the underlying graph of $G_f$ is connected,
  except possibly for some isolated vertices,
  and $\deg^-(s) = \deg^+(s)$ for every vertex $s$,
  except that if $s_1 \ne s_2$,
  then $\deg^-(s_1) = \deg^+(s_1) - 1$ and $\deg^-(s_2) = \deg^+(s_2) + 1$,
\item \label{euler4}
  the underlying graph of $G_f$ is connected,
  except possibly for some isolated vertices, and
  \begin{equation} \label{eq:fsystem}
      \sum_{a \in A} f(as) = \sum_{a \in A} f(sa) + c_s
      \qquad (s \in A^{k - 1}),
  \end{equation}
  where
  \begin{equation*}
      c_s = \begin{cases}
          -1, &\text{if $s = s_1 \ne s_2$}, \\
          1, &\text{if $s = s_2 \ne s_1$}, \\
          0, &\text{otherwise},
      \end{cases}
  \end{equation*}
\end{enumerate}
\end{lemma}

\begin{proof}
\eqref{euler1} $\Leftrightarrow$ \eqref{euler2}:
$w = a_1 \dots a_n \in S(s_1, s_2, n)$ and $f = f_w$ if and only if
\begin{equation*}
  s_1 = a_1 \dots a_{k-1}
  \rightarrow a_2 \dots a_{k}
  \rightarrow \dots
  \rightarrow a_{n - k + 2} \dots a_{n} = s_2
\end{equation*}
is an Eulerian path in $G_f$.

\eqref{euler2} $\Leftrightarrow$ \eqref{euler3}:
This is well known.

\eqref{euler3} $\Leftrightarrow$ \eqref{euler4}:
\eqref{euler4} is just a reformulation of \eqref{euler3}
in terms of the function $f$.
\end{proof}

In the next lemma we consider the independence of homogeneous systems
related to the equations \eqref{eq:fsystem} and \eqref{eq:sumf}.

\begin{lemma} \label{lem:degrsyst}
Let $x_t$, where $t \in A^k$, be $m^k$ unknowns.
The system of equations
\begin{equation} \label{eq:system}
  \sum_{a \in A} x_{as} = \sum_{a \in A} x_{sa}
  \qquad (s \in A^{k-1})
\end{equation}
is not independent, but all of its proper subsystems are.
If we add the equation
\begin{equation} \label{eq:eq}
  \sum_{t \in A^k} x_{t} = 0
\end{equation}
to one of these independent systems, then the system remains independent.
\end{lemma}

\begin{proof}
The sum of the equations \eqref{eq:system} is a trivial identity
\begin{math}
  \sum_{t \in A^k} x_{t} = \sum_{t \in A^k} x_{t},
\end{math}
so every one of these equations follows from the other $m^{k-1}-1$
equations. If $s_1, s_2 \in A^{k-1}$ are two different words,
then
\begin{math}
  x_t = |s_1 s_2|_t
\end{math}
for all $t$ is a solution of all the equations, except those with $s
= s_1$ or $s = s_2$. This proves that all subsystems are
independent. Addition of \eqref{eq:eq} keeps them independent,
because
\begin{math}
  x_t = 1
\end{math}
for all $t$ is a solution of the system \eqref{eq:system} but not of
\eqref{eq:eq}.
\end{proof}

\begin{theorem}
Let $k\geq 1$ and $m \geq 2$ be fixed numbers and let $A$ be an
$m$-letter alphabet. The number of $k$-Abelian equivalence classes
of $A^n$ is
\begin{math}
  \Theta(n^{m^{k} - m^{k - 1}}).
\end{math}
\end{theorem}

\begin{proof}
Let $n \geq 2k-2$,
\begin{math}
  f: A^k \to \{0, \dots, n-k+1\}
\end{math}
and $u,v \in A^{k-1}$. By Lemma \ref{lem:euler}, there is a
word $w \in S(u,v,n)$ such that $f = f_w$ only if $f$ satisfies
\eqref{eq:sumf} and \eqref{eq:fsystem}. Consider the system formed
by these equations. The function $f_w$ satisfies the equations for
every $w \in S(u,v,n)$, so the system has a solution. By Lemma
\ref{lem:degrsyst}, the rank of the coefficient matrix of the system
is $m^{k-1}$, so the general solution of this system is of the form
\begin{equation*}
  f(r_i) = \sum_{j=1}^{m^k - m^{k-1}} a_{ij} f(s_j) + b_i
  \qquad (i = 1, \dots, m^{k-1}),
\end{equation*}
where the words $r_i$ and $s_j$ form the set $A^k$ and $a_{ij},
b_i$ are rational numbers. Because
\begin{math}
  0 \leq f(s_j) \leq n-k+1,
\end{math}
there are
\begin{math}
  O(n^{m^k-m^{k-1}})
\end{math}
possible functions $f$.

Let $u=v$ and consider the system of equations \eqref{eq:fsystem}.
By Lemma \ref{lem:degrsyst}, the general solution of this
homogeneous system is of the form
\begin{equation} \label{eq:sol}
  f(r_i) = \sum_{j=1}^{m^k - m^{k-1} + 1} a_{ij} f(s_j)
  \qquad (i = 1, \dots, m^{k-1}-1),
\end{equation}
where the words $r_i$ and $s_j$ form the set $A^k$ and $a_{ij}$
are rational numbers. The coefficients $a_{ij}$ do not depend on
$n$. Let
\begin{equation*}
  c = \max \set{ \textstyle \sum_{j=1}^{m^k - m^{k-1} + 1}
      |a_{ij}|}{1 \leq i \leq m^{k-1}-1}
\end{equation*}
and let $d$ be the least common multiple of the denominators of the
numbers $a_{ij}$. Every constant function $f$ satisfies the system
of equations. In particular,
\begin{math}
  f(t) = \lfloor {n}/{2m^k} \rfloor
\end{math}
for all $t$ is a solution of the system. If we let
\begin{equation*}
  f(s_j) = \left\lfloor \frac{n}{2m^k} \right\rfloor + b_j,
  \quad \text{where} \quad
  |b_j| < \frac{n}{2cm^k} - 1
  \quad \text{and} \quad
  d | b_j,
\end{equation*}
then the numbers
\begin{equation*}
  f(r_i) = \left\lfloor \frac{n}{2m^k} \right\rfloor
      + \sum_{j=1}^{m^k - m^{k-1} + 1} a_{ij} b_j
\end{equation*}
given by \eqref{eq:sol} are integers and
\begin{math}
  1 \leq f(t) \leq n / m^k - 1
\end{math}
for all $t \in A^k$. Because $f(t) \geq 1$ for all $t$, the
underlying graph of $G_f$ is connected, so by Lemma \ref{lem:euler}
there is a word $w \in S(u, v, |w|)$ such that $f = f_w$. Because
$f(t) \leq n / m^k - 1$ for all $t$, we get
\begin{equation*}
  |w| = \sum_{t \in A^k} f(t) + k - 1
  \leq n - m^k + k - 1 < n .
\end{equation*}
There are
\begin{math}
  \Theta(n^{m^k - m^{k-1} + 1})
\end{math}
ways to choose the numbers $b_j$. Every choice gives a different
function $f = f_w$ for some $w \in S(u, v, |w|)$ such that $|w| <
n$. Let these words be $w_1, \dots, w_N$. No two of them are
$k$-Abelian equivalent. Among these words there are at least $N / n$
words of equal length.
By Lemma \ref{lem:increasing},
there are at least $N / n$ words of length $n$
such that no two of them are $k$-Abelian equivalent,
and $N / n = \Omega(n^{m^k - m^{k - 1}})$.
\end{proof}

\section{$k$-Abelian complexity \& periodicity}\label{periodicity}

In this section we prove that if $\mathcal {P}^{(k)}_\omega (n_0)<q^{(k)}(n_0)$ for some $k\in \ints^+ \cup \{+\infty\}$ and $n_0 \geq 1,$ then $\omega$ is ultimately periodic (see Corollary~\ref{evperiodic} below). For this purpose we introduce an auxiliary family of equivalence relations $\mathcal {R}_k$ on $A^*$ defined as follows: Let $k\in \ints^+ \cup \{+\infty\}.$ Give $u,v \in A^*$ we write $u\mathcal {R}_k v,$ if and only if $u\thicksim _{1} v$ (i.e., $u\thicksim_{ab}v$) and $u$ and $v$ share a common prefix and a common suffix of lengths $k-1.$ In case $|u|<k-1,$ then $u\mathcal {R}_k v$ means  $u=v.$

It follows immediately from Lemma~\ref{prefixsuffix} that

\begin{equation}\label{imp} u\thicksim_{k} v \Longrightarrow u\mathcal {R}_k v.\end{equation}

In general the converse is not true: For example, taking $u=0011$ and $v=0101$ we see that $u\mathcal {R}_2v$ yet $u$ and $v$ are not $2$-Abelian equivalent.
However, in view of Proposition~\ref{simplificationprop} we have:

\begin{corollary} Let $u$ and $v$ be two factors of a Sturmian word $\omega$, and $k\in \ints^+ \cup \{+\infty\}.$ Then $u\thicksim_{k} v$ if and only if $u\mathcal {R}_k v.$
\end{corollary}

Let $\omega \in A^\nats.$ Associated to the relation $\mathcal {R}_k$ is a complexity function, denoted $\p^{(k)}_\omega (n),$ which counts the number of distinct $\mathcal {R}_k$ equivalence classes of factors of $\omega$ of length $n.$ It follows from (\ref{imp}) above that for each $n$ we have

\begin{equation}\label{comps}\p^{(k)}_\omega (n) \leq \mathcal {P}^{(k)}_\omega (n).\end{equation}

We recall the function  $q^{(k)}:\nats \rightarrow \nats$  ($k\in \ints^+ \cup \{+\infty\})$ defined by

\[ q^{(k)}(n) =
\left\{\begin{array}{ll} n+1 \,\,\,&\mbox{for}\,\, n\leq 2k-1\\
2k \,\,\,&\mbox{for}\,\, n\geq 2k
\end{array}
\right.
\]

\begin{theorem}\label{theoremevperiodic2}  Let $\omega=a_0a_1a_2\ldots \in A^\nats$  and
$k \in \ints ^+ \cup \{+\infty\}.$ If $\p_{\omega}^{(k)}(n_0)<q^{(k)}(n_0)$ for some $n_0 \geq 1,$ then $\omega$ is ultimately periodic.
\end{theorem}

\begin{proof} The result is well known in case $k=+\infty$ (see \cite{MorHed1940}). For $k\in \ints^+,$ we proceed by induction on $k.$ In case $k=1,$ then $\mathcal {R}_1$ is simply the usual notion of Abelian equivalence and the result follows from \cite{CovHed}.

Now suppose $k>1$ and that  $\p_{\omega}^{(k)}(n_0)<q^{(k)}(n_0)$ for some $n_0 \geq 1.$ It follows immediately from the definition of $\mathcal{R}_k$ that  if $u\mathcal {R}_kv$ and $|u|\leq 2k-1,$ then $u=v.$ Thus, if $\p_\omega ^{(k)}(n_0)<q^{(k)}(n_0)$ where $n_0\leq 2k-1,$ then $\p_\omega (n_0)<n_0+1$ and so $\omega$ is ultimately periodic by the well known result of Morse and Hedlund in \cite{MorHed1940}.

Thus we suppose that $\p_\omega ^{(k)}(n_0)<2k$ for some $n_0\geq 2k.$  We claim that $\omega$ must be ultimately periodic. Suppose to the contrary that $\omega$ is aperiodic. We shall show that this implies that $\p_\nu ^{(k-1)}(n_0-2)<2(k-1)$ where $\nu=a_0^{-1}\omega$  denotes the first shift of $\omega,$ i.e., the word obtained from $\omega$ by removing the first letter of $\omega.$ Since $n_0-2\geq 2(k-1)$ we deduce that $\p_\nu ^{(k-1)}(n_0-2)<q^{(k-1)}(n_0-2).$  But then by induction hypothesis on $k,$ it follows that $\nu$ (and hence $\omega)$ is ultimately periodic, a contradiction.

Consider the map
\[\Psi : {\mathcal F}_{\omega}(n_0)/\mathcal {R}_k \longrightarrow {\mathcal F}_{\nu}(n_0-2)/\mathcal {R}_{k-1}\]

defined by
\[\Psi ([aub]_k) =[u]_{k-1}\]
where $a,b\in A,$  and $u\in A^*$ of length $n_0-2$ Here $[u ]_k$ denotes the $\mathcal {R}_k$ equivalence class of $u.$
To see that $\Psi$ is well defined, suppose $aub\mathcal {R}_k cud.$ Then since $k>1,$ it follows that $a=c$ and $b=d$ and thus that $u\mathcal {R}_1  v.$ Moreover as $aub$ and $cud$ share a common prefix and suffix of length $k,$ it follows that $u$ and $v$ share a common prefix and suffix of length $k-1.$ Thus $u\mathcal {R}_{k-1} v$ as required.
Clearly the mapping $\Psi$ is surjective, in fact for each $u\in {\mathcal F}_{\nu}(n_0-2)$ there exist $a,b\in A$ such that $aub\in {\mathcal F}_{\omega}(n_0).$ This is the reason for replacing  $\omega$ by $\nu.$

We now show that either there exist distinct classes $[u]_{k-1},[v]_{k-1}\in {\mathcal F}_{\nu}(n_0-2)/\mathcal {R}_{k-1}$ for which
\begin{equation}\label{(*)}\mbox{min}\{\mbox{Card}\left(\Psi^{-1}([u]_{k-1})\right), \mbox{Card}\left(\Psi^{-1}([v]_{k-1})\right)\}\geq 2,\end{equation} or there exists a class $[u]_{k-1}\in {\mathcal F}_{\nu}(n_0-2)/\mathcal {R}_{k-1}$ for which
\begin{equation}\label{(**)}\mbox{Card}\left(\Psi^{-1}([u]_{k-1})\right)\geq 3.\end{equation}
In either case it follows that
\[\mbox{Card}\left({\mathcal F}_{\nu}(n_0-2)/\mathcal {R}_{k-1}\right)\leq\mbox{Card}\left({\mathcal F}_{\omega}(n_0)/\mathcal {R}_k\right) -2< 2(k-1).\]

Since $\omega$ is assumed to be aperiodic, $\omega$ contains both a left special factor of the form $uc$ and a right special factor of the form $dv$ of lengths $n_0-1$ for some choice of $c,d \in A$ and $u,v\in A^*.$ Thus there exist  distinct letters $a,b \in A$ such that $auc$ and $buc$ are factors of $\omega.$ Moreover since $a\neq b,$ it follows that $[auc]_k\neq [buc]_k.$
Thus $\mbox{Card}\left(\Psi^{-1}([u]_{k-1})\right)\geq 2.$
Similarly, there  exist  distinct letters $a',b' \in A$ such that $dva'$ and $dvb'$ are factors of $\omega,$ and since $a'\neq b',$ it follows that $[dva']_k\neq [dvb']_k.$
Thus $\mbox{Card}\left(\Psi^{-1}([v]_{k-1})\right)\geq 2.$ In case $[u]_{k-1}\neq [v]_{k-1},$ we obtain the desired inequality~(\ref{(*)}). In case $[u]_{k-1}= [v]_{k-1},$ since $a\neq b$ and $a'\neq b'$ it follows that
\[\mbox{Card}\{[auc]_k, [buc]_k, [dua']_k,[dub']_k\}\geq 3\]
which yields the inequality~(\ref{(**)})
This completes the proof of Theorem~\ref{theoremevperiodic2}
\end{proof}

\begin{corollary}\label{evperiodic} Let $\omega \in A^\nats$  and
$k \in \ints ^+ \cup \{+\infty\}.$ If $\mathcal {P}^{(k)}_\omega (n_0)<q^{(k)}(n_0)$ for some $n_0 \geq 1$ then $\omega$ is ultimately periodic.
\end{corollary}

\begin{proof} As a consequence of the inequality (\ref{comps}), if $\mathcal {P}^{(k)}_\omega (n_0)<q^{(k)}(n_0)$ then
$\p_{\omega}^{(k)}(n_0)<q^{(k)}(n_0),$ whence by Theorem~\ref{theoremevperiodic2} it follows that $\omega$ is ultimately periodic.
\end{proof}

\noindent The same method of proof of Theorem~\ref{theoremevperiodic2} can be used to prove the following:

\begin{corollary}\label{theoremperiodic} Let $\omega $ be a bi-infinite word over the alphabet $A$ and
$k \in \ints ^+ \cup \{+\infty\}.$ If $\mathcal {P}^{(k)}_\omega (n_0)<q^{(k)}(n_0)$ for some $n_0 \geq 1,$ then $\omega$ is periodic.
\end{corollary}

\noindent We conclude this section with a few remarks:

\begin{remark} \rm{ In the special case $k=+\infty,$ the condition given in Corollary~\ref{evperiodic} gives a characterization of ultimately periodic words by means of factor complexity: $\omega \in A^\nats$ is ultimately periodic if and only if $\p_{\omega}(n_0)<n_0+1$ for some $n_0\geq 1.$
However, $k$-Abelian complexity does not yield such a characterization. Indeed, both Sturmian words and
the ultimately periodic word  $01^\infty = 0111\cdots$ have the same Abelian complexity. More generally, the ultimately periodic word
$0^{2k-1}1^\infty\ldots$  has the same $k$-Abelian complexity as a Sturmian word (see Theorem~\ref{thm:factcompl} below).}
\end{remark}

\begin{remark} \rm{The result of Corollary~\ref{theoremperiodic} is already known to be true in the special cases $k=+\infty$ (see \cite{MorHed1940}) and $k=1$ (see  Remark 4.07 in \cite{CovHed}). In these special cases, the converse is also true. But for general $2\leq k<+\infty $ the converse is false. For instance, let $\mbox{Card}(A)=5,$ and let $u$ be a word containing at least one occurrence of every $x\in A^3.$ Let $\omega$ be the periodic word $\omega =\ldots uuuu\ldots.$ Then $\p^{(2)}_\omega (n)\geq 5$ for every $n\geq 1.$ }
\end{remark}

\section{$k$-Abelian complexity of Sturmian words}\label{Sturmwords}

In this section we determine the $k$-Abelian complexity of Sturmian words and show that for each $k,$ the complexity function $\mathcal {P}^{(k)}$ completely characterizes Sturmian words amongst all aperiodic words. More precisely:

\begin{theorem}\label{thm:factcompl}
Fix $k\in \ints^+ \cup \{+\infty\}.$ Let  $\omega \in A^\nats$ be an aperiodic word. The following conditions are equivalent:
\begin{itemize}
\item  $\omega$ is a balanced binary word, that is, {\it Sturmian}.
\item \begin{math}
  \fc{k}{\omega}(n) = q^{(k)}(n) =
  \begin{cases}
      n + 1 & \text{for $0 \leq n \leq 2k - 1$} \\
      2k & \text{for $n \geq 2k$}
  \end{cases}.
\end{math}
\end{itemize}
\end{theorem}

Our proof of Theorem~\ref{thm:factcompl} will make use of the following functions $\swap{i}$, which transform binary words by changing
the letters around a specific point. For words $w \in \{0,1\}^n$ we
define $\swap{1}, \dots, \swap{n}$ as follows:
\begin{equation*}
  \swap{i}(w) = \begin{cases}
  u10v, &\text{if $i < n$, $w = u01v$ and $|u0| = i$}, \\
  u1, &\text{if $i = n$ and $w = u0$}.
  \end{cases}
\end{equation*}

\begin{lemma} \label{lem:facttrans}
Let $n \geq 1$ and let $w \in \{0,1\}^{\omega}$ be Sturmian. There
is a word $u_1 \in \{0,1\}^n$ and a permutation $\sigma$ of $\{1,
\dots, n\}$ such that if $u_{i+1} = \swap{\sigma(i)}(u_i)$ for $i =
1, \dots, n$, then $u_1, \dots, u_{n+1}$ are the factors of $w$ of
length $n$.
\end{lemma}

\begin{proof}
Let $u_1, \dots, u_{n + 1}$ be the factors of $w$ of length $n$
in lexicographic order.
If follows from Theorem 1.1. in \cite{BuLuZa12}
that for every $i$ there is an $m$ such that $u_{i + 1} = \swap{m}(u_i)$.
It needs to be proved that the $m$'s are all different.
Let $u_{i + 1} = \swap{m}(u_i)$ and $u_{i' + 1} = \swap{m}(u_i')$.
For every $j$
\begin{equation*}
  |\pref{m}(u_j)|_1 \leq |\pref{m}(u_{j + 1})|_1
\end{equation*}
and for $j \in \{i, i'\}$
\begin{equation*}
  |\pref{m}(u_j)|_1 < |\pref{m}(u_{j + 1})|_1 .
\end{equation*}
If $i \ne i'$, then
\begin{equation*}
  |\pref{m}(u_1)|_1 + 2 \leq |\pref{m}(u_{n + 1})|_1
\end{equation*}
which contradicts the balance property of Sturmian words.
\end{proof}

\begin{example}
The factors of the Fibonacci word of length six are
\begin{alignat*}{3}
  u_1 = 001001,
        \quad u_2 &= 001010 = g_5(u_1),
      & \quad u_3 &= 010010 = g_2(u_2),
      & \quad u_4 &= 010100 = g_4(u_3), \\
  u_5 &= 100100 = g_1(u_4),
      & u_6 &= 100101 = g_6(u_5),
      & u_7 &= 101001 = g_3(u_6).
\end{alignat*}
We have $u_2 \kae{2} u_3 \kae{2} u_4$ and $u_6 \kae{2} u_7$. There
are no other 2-Abelian equivalences between these factors.
\end{example}

\begin{proof}[Proof of Theorem~\ref{thm:factcompl}]
First let us suppose $\omega \in \{0,1\}^\nats$ is Sturmian and let $1\leq k \leq+\infty.$
Let $n \leq 2k-1.$
By Lemma~\ref{lem:basic} two factors $u$ and $v$ of $\omega$ of length $n$ are  $k$-Abelian equivalent if and only $u=v.$ Thus
$\fc{k}{w}(n) = n+1$ as required.

Next let $n \geq 2k$ and let $u_1, \dots, u_{n+1}$ and $\sigma$ be as
in Lemma \ref{lem:facttrans}. If $k \leq \sigma(i) \leq n-k$, then
there are words $s, t \in \{0,1\}^*$ and $u, v \in \{0,1\}^{k-1}$
and letters $a, b \in \{0,1\}$ so that $u_i = su01vt$ and $u_{i+1} =
\swap{\sigma(i)}(u_i) = su10vt$. We prove that $u_i \kae{k}
u_{i+1}$. The prefixes and suffixes of $u_i$ and $u_{i+1}$ of length
$k-1$ are the same. The factors of $u_i$ of length $k$ are the
factors of $su$, $u01v$ and $vt$ of length $k$, and the factors of
$u_{i+1}$ of length $k$ are the factors of $su$, $u10v$ and $vt$ of
length $k$.
Because $u01v$ and $u10v$ are factors of $w$,
it follows that $u$ is right special and
$v$ is left special and hence equal to the reversal of $u$.
By Theorem~\ref{2k}, $u01v$ and $u10v$ are $k$-Abelian equivalent.
This proves that $u_i \kae{k} u_{i+1}$ if $k \leq \sigma(i) \leq n-k$.
Thus the words $u_1, \dots, u_{n+1}$
are in at most $2k$ different $k$-Abelian equivalence classes
and $\fc{k}{\omega}(n) \leq 2k$.
By Corollary \ref{evperiodic}, $\fc{k}{\omega}(n) = 2k$.

Next let $1\leq k\leq +\infty$ and let $\omega \in A^\nats$  be aperiodic  and
\begin{equation*}
  \fc{k}{\omega}(n) = q^{(k)}(n) =
  \begin{cases}
      n + 1 & \text{for $0 \leq n \leq 2k - 1$} \\
      2k & \text{for $n \geq 2k$}
  \end{cases}.
\end{equation*}

\noindent Taking $n=1$ we see that $\omega$ is binary, (say $\omega \in \{0,1\}^\nats).$ We must show that $\omega$ is balanced.
We first recall some basic facts concerning factors of Sturmian words (see for instance \cite{RiZa}):
Let $\eta \in \{0,1\}^\nats$ be a Sturmian word, and let ${\mathcal F}_{\eta}(n)$ denote the factors of $\eta$ of length $n.$
The set ${\mathcal F}_{\eta}(n+1)$ is completely determined from the set ${\mathcal F}_{\eta}(n)$ unless $\eta$ has a bispecial factor $B$ of length $n-1$ in which case both $0B$ and $1B$ are factors of $\eta$ and exactly one of the two is right special.  If $0B$ is right special, then every occurrence of $1B$ in $\eta$ is an occurrence of $1B0.$ If $v$ is a factor of $\eta$ and $u$ a prefix of $v,$ we write $u\vdash v$ if every occurrence of $u$ in $\eta$ is an occurrence of $v.$ Thus if $0B$ is right special, then $1B\vdash 1B0,$ and similarly if $1B$ is right special, then $0B \vdash 0B1.$

Now suppose to the contrary that the aperiodic binary word $\omega$ is not Sturmian. Then there exists a smallest positive integer $n\geq 1$ and a Sturmian word $\eta$ such that ${\mathcal F}_{\omega}(n)= {\mathcal F}_{\eta}(n)$ but
${\mathcal F}_{\omega}(n+1)\neq {\mathcal F}_{\eta'}(n+1)$ for every choice of Sturmian word $\eta'.$
This means that $\omega$ has a bispecial factor $B$ of length $n-1$ and both $0B$ and $1B$ are in
${\mathcal F}_{\omega}(n)$ and one of the following must occur: i) Neither $0B$ nor $1B$ is right special in $\omega;$ ii) There exists a unique $a\in \{0,1\}$ such that $aB$ is right special, and $(1-a)B\vdash (1-a)B(1-a);$  iii) Both $0B$ and $1B$ are right special in $\omega.$ We will show that since $\omega$ is aperiodic, only case iii) is in fact possible. Clearly, if neither $0B$ nor $1B$ were right special, then $\mbox{Card}({\mathcal F}_{\omega}(n))=\mbox{Card}({\mathcal F}_{\omega}(n+1))$ whence $\omega$ is ultimately periodic, a contradiction.
Next suppose case ii) occurs. We may suppose without loss of generality that $0B$ is right special and $1B
\vdash 1B1.$ If $1\vdash 1B$ (and hence $1\vdash 1B1),$ then we would have  $1\vdash 1(B1)^n$ for every $n\geq 1$ from which it follows that  the tail of $\omega$ corresponding to the first occurrence of $1$ on $\omega$ is periodic. Thus if $\neg(1\vdash 1B),$ then there exists a bispecial factor $B'$ of $\omega$  with $0<|B'|<|B|$ such that $1B'$ is right special and $1B'1 \vdash 1B$ and hence $1B'1 \vdash 1B1.$  Writing $1B1=1B'1V$ we have $1B'1 \vdash 1B'1V.$ We next show by induction on $n$ that $1B'1V^n$ is a palindrome for each $n\geq 1.$  Clearly this is true for $n=1$ since $1B'1V=1B1.$
Next suppose $1B'1V^n$ is a palindrome. Then
\[\overline{1B'1V^{n+1}}=\overline{V}^{n+1}\overline{1B'1}=\overline{V}\,\overline{V}^n\overline{1B'1}=\overline{V}1B'1V^n=\overline{V}\,\overline{1B'1}\,V^n=1B'1VV^{n}=1B'1V^{n+1}.\]
Having established that $1B'1V^n$ is a palindrome, it follows that $1B'1$ is a suffix of $1B'1V^n$ and hence $1B'1V^n\vdash 1B'1V^{n+1}$ for each $n\geq 0.$
Whence as before $\omega $ is ultimately periodic.
Thus if $\omega$ is not Sturmian, case iii) must occur. This implies that
\[{\mathcal F}_{\omega}(n+1)= {\mathcal F}_{\eta}(n+1)\cup \{0B0,1B1\}\]
and $\mbox{Card}({\mathcal F}_{\eta}(n+1)\cap \{0B0,1B1\})=1.$ Since $\eta$ is Sturmian,
the number of $k$-Abelian classes of factors of $\eta$ of length $n+1$ is equal to $q^{(k)}(n+1).$
But the additional factor $aBa$ of $\omega$ of length $n+1$ introduces a new $k$-Abelian class since it is not even Abelian equivalent to any other factor of $\eta$ (and hence $\omega)$ of length $n+1.$ Thus $\fc{k}{\omega}(n+1) = q^{(k)}(n+1)+1,$ a contradiction. Thus $\omega$ is Sturmian.

\end{proof}


  \begin{remark} \rm{In view of Corollary~\ref{evperiodic}, within the class of aperiodic words, Sturmian words have the lowest possible $k$-Abelian complexity. See \cite{AFKS,KWZ,KZ, MorHed1940} for other instances in which Sturmian words have the lowest complexity amongst all aperiodic words.}
\end{remark}

\section{Bounded $k$-Abelian complexity \& $k$-Abelian repetitions}\label{last}

There is great interest in avoidability of repetitions in infinite words. This originated with the classical work of Thue \cite{Th06} and \cite{Th12}, in which he established the existence of an infinite binary (resp. ternary) word avoiding cubes (resp. squares).
It was later shown that to avoid Abelian cubes or Abelian squares, one needs $3$-letter or $4$-letter alphabets respectively (see \cite{Dekking1979} and \cite{Keranen1992ICALP}).
The corresponding problems for $k$-abelian repetitions turned out to be quite nontrivial.
It follows easily that the smallest alphabet where $k$-abelian cubes can be avoided is either $2$ or $3,$ and similarly the smallest alphabet where $k$-abelian squares can be avoided is either $3$ or $4.$
In the latter case for $k = 2$ a computer verification revealed that the correct value is $4,$ as in the case of Abelian repetitions:
Each ternary $2$-abelian square-free word is of length at most $536$ \cite{HuKa11}.
In the former case computer verification shows that there exist binary words of length $100000$ which are $2$-abelian cube-free \cite{HuKa11}.
It is still unknown whether there exists an infinite binary word which is $2$-abelian cube-free.
For some larger values of $k$ such infinite words exist.
In the case of binary alphabets and cubes it was shown in a sequence of papers that an infinite word avoiding $k$-abelian cubes can be constructed for $k = 8$, $k = 5$ and for $k = 3$ (see \cite{HuKaSa12ehrenfeucht}, \cite{MeSa12jm} and \cite{MeSa13dlt} respectively).
So only the value $k = 2$ remains open.
It would be extremely surprising if no such infinite words exist.
For avoiding $k$-abelian squares in a ternary alphabet the situation is equally challenging.
We know that for $k = 3$ there exist words of length $100000$ avoiding $3$-abelian squares. The avoidability in infinite words of $k$-abelian squares in a ternary alphabet  is only known for large values of $k$ ($k\geq 64)$ (see \cite{Huova}).

In this section we prove that $k$-Abelian repetitions are unavoidable in words having bounded $k$-Abelian complexity.
For each positive integer $k$ we set
\[A^{\leq k} =\{ x\in A^*\,:\, |x|\leq k\}.\] Given an infinite word $\omega = a_0a_1a_2\ldots \in A^\nats,$ for each
$0\leq i<j <+\infty$ we denote by $\omega[i,j]$ the factor $a_ia_{i+1}\cdots a_j.$

\begin{definition} Let $k$ and $B$ be positive integers and $\omega \in A^\nats.$ We say $\omega$ is $(k,B)$-balanced if and only if for all factors $u$ and $v$ of $\omega$ of equal length, and for all $x \in A^{\leq k}$ we have $\left||u|_x-|v|_x\right| \leq B.$ We say $\omega$ is arbitrarily $k$-imbalanced if $\omega$ is not $(k,B)$-balanced for any positive integer $B.$
\end{definition}

\noindent An elementary, but key observation is that

\begin{lemma}\label{balance} Let $k$ be a positive integer and $\omega \in A^\nats.$ Then $\omega$ has bounded  $k$-Abelian complexity if and only if $\omega$ is $(k,B)$-balanced for some positive integer $B.$
\end{lemma}

\begin{proof} Clearly if  $\mathcal {P}^{(k)}_\omega$ is bounded, say by $B,$ then $\omega$ is $(k,B-1)$-balanced. Conversely, if $\omega$ is $(k,B)$-balanced, then for each positive integer $n$ and for each $x\in A^*$ with $|x|\leq k$ we have
\[\mbox{Card}\{ |u|_x:\, u \in {\mathcal F}_{\omega}(n)\} \leq B+1.\]
It follows that
\[\mathcal {P}^{(k)}_\omega (n) \leq (B+1) ^{K}\]
where $K=\mbox{Card}A^{\leq k}.$

\end{proof}

Fix a positive integer $k.$ It follows from Theorem~\ref{thm:factcompl} and Lemma~\ref{balance} that each Sturmian word is $(k,B)$-balanced for some positive integer $B$ (depending on $k.)$ Actually, I. Fagnot and L. Vuillon proved in \cite{FV} that every Sturmian word is $(k,k)$-balanced.
\begin{definition} Fix $k\in \ints^+ \cup \{+\infty\},$ and $N$ a positive integer. By a $k$-Abelian $N$-power we mean a word $U$ of the form $U=U_1U_2\cdots U_N$ such that $U_i\thicksim_k U_j$ for all $1\leq i,j\leq N.$
\end{definition}

\noindent In this section we shall prove the following result:

\begin{theorem}\label{kabpowers}
Fix $k\in \ints^+ \cup \{+\infty\}.$ Let $\omega =a_0a_1a_2\ldots \in A^\nats$  be an infinite word on a finite alphabet $A$ having bounded $k$-Abelian complexity. Let $D\subseteq \nats$
be a set of  positive upper density, that is
\[
\limsup_{n\rightarrow \infty} \frac{\mbox{Card}\left(D \cap \{1,2, \ldots, n\} \right) }{n} >0.
\]
Then, for every positive integer $N$, there exist $i$ and $\ell$ such that $\{i, i+\ell, i+2\ell, \ldots, i+\ell N\}\subset D$  and the $N$ consecutive blocks  $(\omega[i+j\ell, i+(j+1)\ell-1])_{0\leq j\leq N-1}$ of length $\ell $ are pairwise $k$-Abelian equivalent.  In particular, $\omega$ contains arbitrarily high $k$-Abelian powers.

\end{theorem}


\begin{remark}\rm{The result in Theorem~\ref{kabpowers} is already known in the special case of $D=\nats$ and $k=+\infty$ and  $k=1$ (see \cite{MorHed1940} and \cite{RSZ2} respectively).}
\end{remark}

\noindent Before proving Theorem~\ref{kabpowers} we give some immediate consequences:

\begin{corollary} Let $k$ and $N$ be positive integers, and $\omega $ an infinite word avoiding $k$-Abelian $N$-powers. Then $\omega$ is arbitrarily $k$-imbalanced.
\end{corollary}

\begin{proof} This follows immediately from Lemma~\ref{balance} and Theorem~\ref{kabpowers}.
\end{proof}

\begin{corollary}\label{sturmpowers} Let $\omega$ be a Sturmian word. Then $\omega$ contains  $k$-Abelian $N$-powers for all positive integers $k$ and $N.$
\end{corollary}

\begin{proof} This follows immediately from Theorems~\ref{thm:factcompl} and \ref{kabpowers}; in fact  the $k$-Abelian complexity $\mathcal {P}^{(k)}_\omega$ is bounded (by $2k$) for each positive integer $k.$
\end{proof}

\begin{remark}\rm{It is known that a Sturmian word $\omega$ contains an $N$-power for each positive integer $N$ if and only if the sequence of partial quotients in the continued fraction expansion of the slope of $\omega$ is unbounded. So, a Sturmian word whose corresponding slope has bounded partial quotients (e.g., the Fibonacci word) will not contain $N$-powers for $N$ sufficiently large (e.g.,  the Fibonacci word contains no $4$-powers \cite{Kar,MiPi}). However, every Sturmian word will contain arbitrarily high $k$-Abelian powers. }
\end{remark}

Our proof of Theorem~\ref{kabpowers} will make use of the following well known result
first conjectured by Erd\"os and Turan and later proved by to E. Szemer\'edi:

\begin{theorem}\label{vdw}{\rm[Szemer\'edi's theorem \cite{Sz}]}
Let $D\subseteq \nats$ be a set of positive upper density. Then $D$ contains arbitrarily long arithmetic progressions.
\end{theorem}

\begin{proof}[Proof of Theorem~\ref{kabpowers}] Let $D\subseteq \nats$ be a set of positive upper density. First we consider the case $k=+\infty.$ By assumption
$ \mathcal {P}^{(+\infty)}_\omega (n)$ is bounded. This is equivalent to saying that $\omega$ has bounded factor complexity. It follows by Morse-Hedlund that $\omega$ is ultimately periodic, i.e., $\omega =UV^\infty$ for some $U,V\in A^*.$  For each $i\geq 0,$ set $D_i= D\cap \{i+j|V|\,:\, j=1,2,3,\ldots \}.$ Pick $i>|U|$ such that the set $D_i$ has positive upper density.  Then an arithmetic progression of length $N + 1$ in $D_i$ (guaranteed by Szemer\'edi's theorem) determines the $N$th power of some cyclic conjugate of $V.$

Next let us fix positive integers $k$ and $N$ and assume that $ \mathcal {P}^{(k)}_\omega (n)$ is bounded.  It follows by Lemma~\ref{balance} that $\omega$ is $(k,B)$-balanced for some positive integer $B.$  We recall the following lemma proved in \cite{RSZ2}

\begin{lemma}\label{LA}{\rm[Lemma~5.4 in \cite{RSZ2}]} Let $k$ and $B$ be positive integers. There exist positive integers $\alpha _x$ for each $x\in A^{\leq k}$ and a positive integer $M$ such that whenever
\[\sum _{x\in A^{\leq k}}c_x\alpha _x \equiv 0 \pmod {M}\]
for integers $c_x$ with $|c_x|\leq B$ for each $x\in A^{\leq k},$ then
$c_x=0$ for each $x\in A^{\leq k}.$
\end{lemma}

Set
\[\mathcal{D}= (D-1)\cap \{k,k+1,k+2\ldots\}.\]
Then $\mathcal{D}$ is of positive upper density.
We now define a finite coloring
\[\Phi: \mathcal{D} \longrightarrow \{0,1,2,\ldots , M-1\}\times {\mathcal F}_{\omega}(2k) \]
\noindent as follows
\[\Phi (n) \doteqdot \left(\sum_{x\in A^{\leq k}}|\omega[1,n]|_x\alpha_x \,(\bmod {M})\,; \omega[n-k+1,n+k]\right)\]

\noindent where $\alpha_x$ and $M$ are as in Lemma~\ref{LA}. Note that the second coordinate of $\Phi(n)$ is the suffix of length $2k$ of $\omega[1,n+k].$ We note also that if $\Phi(m)=\Phi(n)$ for some $m<n,$ then by considering the first coordinate of $\Phi$ one has

\begin{eqnarray}
\sum_{x\in A^{\leq k}}|\omega[1,n]|_x\alpha_x - \sum_{x\in A^{\leq k}}|\omega[1,m]|_x\alpha_x \equiv 0 \pmod{M}
\end{eqnarray}

\begin{eqnarray}
\sum_{x\in A^{\leq k}}\left(|\omega[1,n]|_x - |\omega[1,m]|_x\right)\alpha_x \equiv 0 \pmod{M}
\end{eqnarray}

\begin{eqnarray}\label{used}
\sum_{x\in A^{\leq k}} |\omega[m-|x|+2, n]|_x\alpha_x \equiv 0 \pmod{M}.
\end{eqnarray}

$\Phi$ defines a finite partition of $\mathcal{D}$ where two elements $r$ and $s$ in $\mathcal{D}$ belong to the same class of the partition if and only if $\Phi(r)=\Phi(s).$ Clearly at least one class of this partition of $\mathcal{D}$ has positive upper density. Thus by Szemer\'edi's theorem, there exist positive integers $r$ and $t$ with $r\geq k$ such that
\[\{r,r+t,r+2t,\ldots,r+Nt\}\subset \mathcal{D}\] and
\[\Phi(r)=\Phi(r+t)=\Phi(r+2t)=\cdots =\Phi(r+Nt).\]

We now claim that the $N$ consecutive blocks of length $t$
\[\omega[r+1,r+t]\omega[r+t+1,r+2t]\omega[r+2t+1,r+3t]\ldots \omega[r+(N-1)t+1, r+Nt]\]
are pairwise $k$-Abelian equivalent. This would prove that $\omega$ contains a $k$-Abelian $N$-power in position $r+1\in D.$

To prove the claim, let $0\leq i,j \leq N-1.$ We will show that
\[\omega[r+it+1,r+(i+1)t]\thicksim_k \omega[r+jt+1, r+(j+1)t].\]

\noindent By (\ref{used}) first taking $n=r+(i+1)t$ and $m=r+it,$ then $n=r+(j+1)t$ and $m=r+jt$
\[\sum_{x\in A^{\leq k}}|\omega[r+it-|x|+2,r + (i+1)t]|_x\alpha_x \equiv  \sum_{x\in A^{\leq k}}|\omega[r+jt-|x|+2,r + (j+1)t]|_x\alpha_x \equiv 0 \pmod{M}\]

\noindent and hence
\[\sum_{x\in A^{\leq k}}\left(|\omega[r+it-|x|+2,r + (i+1)t]|_x-|\omega[r+jt-|x|+2,r + (j+1)t]|_x\right)\alpha_x \equiv 0 \pmod{M}.\]

\noindent But since
\[|\omega[r+it-|x|+2,r + (i+1)t]|= |\omega[r+jt-|x|+2,r + (j+1)t]|=|x|+t-1\]
and $\omega$ is $(k,B)$-balanced, it follows that
\[\lvert|\omega[r+it-|x|+2,r + (i+1)t]|_x-|\omega[r+jt-|x|+2,r + (j+1)t]|_x\rvert\leq B\]
\noindent whence by Lemma~\ref{LA} we deduce that for each $x\in A^{\leq k}$
\begin{eqnarray}\label{phew}|\omega[r+it-|x|+2,r + (i+1)t]|_x=|\omega[r+jt-|x|+2,r + (j+1)t]|_x.\end{eqnarray}

\noindent Since $\Phi(r+it)=\Phi(r+jt),$ the second coordinate of $\Phi$ gives
\[\omega[r+it -k+1,r+it+k]=\omega[r+jt -k +1, r+jt+k].\]
Together with (\ref{phew}) we deduce that  for each $x\in A^{\leq k}.$
\[ |\omega[r+it+1,r + (i+1)t]|_x=|\omega[r+jt+1,r + (j+1)t]|_x.\]

\noindent In other words
\[\omega[r+it+1,r + (i+1)t]\thicksim_k \omega[r+jt+1,r + (j+1)t]\]
as required. This completes our proof of Theorem~\ref{kabpowers}

\end{proof}


\begin{thebibliography}{50}

\bibitem{AFKS}
S.V.~Avgustinovich, A.~Frid, T.~Kamae, P.~Salimov,
\newblock Infinite permutations of lowest maximal pattern complexity,
\newblock {\em Theoret. Comput. Sci.,} 412 (2011) 2911--2921.

\bibitem{BuLuZa12}
M.~Bucci, A.~De Luca, L.Q.~Zamboni,
\newblock Some characterizations of Sturmian words in terms of the lexicographic order,
\newblock {\em Fund. Inform.}, 116 (2012)  25--33.

\bibitem{CRSZ}
J.~Cassaigne, G.~Richomme, K.~Saari, L.Q. Zamboni,
\newblock Avoiding Abelian powers in binary words with bounded Abelian complexity,
\newblock {\em Internat. J. Found. Comput. Sci.}, 22 (2011) 905--920.

\bibitem{CovHed}
E.~M. Coven, G.~A. Hedlund,
\newblock Sequences with minimal block growth,
\newblock {\em Math. Systems Theory}, 7 (1973)  138--153.

\bibitem{CurRam2009}
J.~Currie, N.~Rampersad.
\newblock Recurrent words with constant Abelian complexity,
\newblock {\it Adv. in Appl. Math.,} 47 (2011) 116--124.

\bibitem{Dekking1979}
F.M. Dekking,
\newblock Strongly non-repetitive sequences and progression-free sets,
\newblock {\em J. Combin. Theory Ser. A,} 27 (1979) 181--185.

\bibitem{DJP} X. Droubay,  J. Justin, G. Pirillo,
\newblock  Episturmian words and some constructions of de Luca and Rauzy, 
\newblock {\it Theoret. Comput. Sci.,} 255 (2001) 539--553.

\bibitem{ER} A. Ehrenfeucht, G. Rozenberg,
\newblock Elementary homomorphisms and a solution of the D0L sequence equivalence problem,
\newblock {\it Theoret. Comput. Sci.,} 7 (1978) 169--183.



\bibitem{FV}
I.~Fagnot, L.~Vuillon,
\newblock Generalized balances in Sturmian words,
\newblock {\it Discrete Appl. Math.,} 121 (2002)  83--101.

\bibitem{FDFF}
D.G.~Fon-Der-Flaass, A.~Frid,
\newblock On periodicity and low complexity of infinite permutations,
\newblock {\it European J. Combin.,} 28 (2007) 2106--2114.

\bibitem{Huova}
M. Huova,
\newblock Existence of infinite ternary $k$-Abelian square free words,
\newblock Preprint, 2013.

\bibitem{HuKa11}
M. Huova, J. Karhum{\"a}ki.
\newblock Observations and problems on $k$-abelian avoidability,
\newblock In {\em Combinatorial and Algorithmic Aspects of Sequence Processing
 (Dagstuhl Seminar 11081)}, (2011) 2215--2219.

\bibitem{HuKaSa12ehrenfeucht}
M. Huova, J. Karhum{\"a}ki,  A. Saarela,
\newblock Problems in between words and abelian words: $k$-abelian
 avoidability,
\newblock {\em Theoret. Comput. Sci.} 454 (2012) 172--177.

\bibitem{KWZ}
T.~Kamae, S.~Widmer, L.Q.~Zamboni,
\newblock Maximal pattern Abelian complexity,
\newblock Preprint, 2013.

\bibitem{KZ}
T.~Kamae, L.Q.~Zamboni,
\newblock Sequence entropy and the maximal pattern complexity of
infinite words,
\newblock {\it Ergodic Theory Dynam. Systems,} 22 (2002) 1191--1199.

\bibitem{Ka80}
J. Karhum{\"a}ki,
\newblock Generalized Parikh mappings and homomorphisms,
\newblock {\em Information and Control,} 47 (1980) 155--165.

\bibitem{Kar}
J. Karhum{\"a}ki,
\newblock On cube free $\omega$-words generated by binary morphisms,
\newblock {\it Discrete Appl. Math.,} 5 (1983)  279--297.


\bibitem{Keranen1992ICALP}
V.~Ker\"anen.
\newblock Abelian squares are avoidable on $4$ letters.
\newblock In W.~Kuich, editor, {\em Proceedings of ICALP'1992 (International
Conference on Automata, Languages and Programming - Vienna 1992)}, volume 623
of {\em Lecture Notes in Comput. Sci.}, pages 41--52. Springer, Berlin, 1992.

\bibitem{Lothaire1983book}
M.~Lothaire,
\newblock {\em Combinatorics on {W}ords}, volume~17 of {\em Encyclopedia of
Mathematics and its Applications}.
\newblock Addison-Wesley, 1983.
\newblock Reprinted in the {\em Cambridge Mathematical Library}, Cambridge
University Press, UK, 1997.


\bibitem{MeSa12jm}
R. Merca{\c s}, A. Saarela,
\newblock $5$-abelian cubes are avoidable on binary alphabets,
\newblock In {\em Proceedings of the 14th Mons Days of Theoretical Computer
 Science}, 2012.

\bibitem{MeSa13dlt}
R. Merca{\c s}, A. Saarela,
\newblock $3$-abelian cubes are avoidable on binary alphabets, Preprint 2013.

\bibitem{MiPi}
F. Mignosi, G. Pirillo,
\newblock Repetitions in the Fibonacci infinite word,
\newblock {\it RAIRO  Theor. Inform. Appl.,} 26 (1992) 199-204.

\bibitem{MorHed1940}
M.~Morse, G.A. Hedlund,
\newblock Symbolic {D}ynamics {II}: {S}turmian trajectories,
\newblock {\em Amer. J. Math.,} 62 (1940)  1--42.

\bibitem{Post}
E. Post,
\newblock A variant of a recursively unsolvable problem,
\newblock {\it Bull. Amer. Math. Soc.,} 52 (1946) 264--268.

\bibitem{PZ}
S. Puzynina, L.Q. Zamboni,
\newblock Abelian returns in Sturmian words, 
\newblock {\it J. Combin. Theory Ser. A,} 120 (2013)  390-408.




\bibitem{RSZ1}
G.~Richomme, K.~Saari, L.Q. Zamboni,
\newblock Balance and Abelian complexity of the {T}ribonacci word,
\newblock {\em Adv. Appl. Math.,}, 45 (2010)  212--231.


\bibitem{RSZ2}
G.~Richomme, K.~Saari,  L.Q. Zamboni,
\newblock Abelian complexity of minimal subshifts,
\newblock {\em J. London Math. Soc. (2),} 83 (2011) 79--95.

\bibitem {RiZa}
R.~ Risley, L.Q. Zamboni.
\newblock A generalization of Sturmian sequences; combinatorial
structure and transcendence,
\newblock {\em Acta Arith.,} XCV.2 (2000) 167--184.

\bibitem{aleksi}
A.~Saarela,
\newblock Ultimately constant abelian complexity of infinite words,
\newblock {\em J. Autom. Lang. Comb.,} 14 (2009)  255--258.

\bibitem{Sz}
E.~Szemer\'edi,
\newblock On sets of elements containing no $k$ elements in arithmetic progressions,
\newblock {\em Acta Arith.,} 27 (1975) 299--345.

\bibitem{Th06}
A. Thue,
\newblock {\"U}ber unendliche zeichenreihen,
\newblock {\em Norske Vid. Selsk. Skr. I. Mat. Nat. Kl.,} 7 (1906) 1--22.

\bibitem{Th12}
A. Thue,
\newblock {\"U}ber die gegenseitige lage gleicher teile gewisser
 zeichen-reihen,
\newblock {\em Norske Vid. Selsk. Skr. I. Mat. Nat. Kl.,} 1 (1912) 1--67.


\end{thebibliography}
\end{document}